\documentclass[letterpaper, 10 pt, conference]{ieeeconf}  

\IEEEoverridecommandlockouts                              

\def\urls#1{{\small\url{#1}}}

\def\bfmath#1{{\mathchoice{\mbox{\boldmath$#1$}}%
{\mbox{\boldmath$#1$}}%
{\mbox{\boldmath$\scriptstyle#1$}}%
{\mbox{\boldmath$\scriptscriptstyle#1$}}}}

\def\rd#1{{\color{red}#1}}

\usepackage[colorlinks=true,bookmarks=false,pdfborder={0 0 0},
    linkcolor=black,urlcolor=black,citecolor=black,
    breaklinks=false,bookmarksnumbered=false]{hyperref}
\usepackage{cite} 
\usepackage[utf8]{inputenc}
\usepackage{graphicx}
\graphicspath{ {images/} }
\usepackage{array}
\newcolumntype{P}[1]{>{\centering\arraybackslash}p{#1}}
\usepackage{amsmath}
\usepackage{multirow}
\usepackage{array}
\usepackage{dsfont}
\usepackage{tikz}
\usepackage{amssymb}
\usepackage{amsthm}
\usepackage[ ruled]{algorithm2e}
\usepackage{algcompatible}
\usepackage{float}
\usepackage{multicol}

\newcommand{\utwi}[1]{\mbox{\boldmath $#1$}}

\renewcommand{\hat}{\widehat}
\renewcommand{\tilde}{\widetilde}
\newcommand{\epsy}{\epsilon}

\newcommand{\cD}{{\cal D}}

\newcommand{\cL}{{\cal{L}}}
\newcommand{\cN}{{\cal N}}

\newcommand{\cT}{{\cal T}}

\newcommand{\cX}{{\cal X}}

\newcommand{\bg}{{\bfmath{g}}}
\newcommand{\bh}{{\bfmath{h}}}
\newcommand{\bl}{{\bfmath{\ell}}}

\newcommand{\bx}{{\bfmath{x}}}

\newcommand{\bv}{{\bf v}}
\newcommand{\bw}{{\bfmath{w}}}

\newcommand{\bz}{{\bfmath{z}}}
\newcommand{\by}{{\bfmath{y}}}

\newcommand{\bC}{{\bfmath{C}}}
\newcommand{\bD}{{\bfmath{D}}}

\newcommand{\bJ}{{\bfmath{J}}}

\newcommand{\bI}{{\bfmath{I}}}

\newcommand{\bbeta}{{\utwi{\beta}}}

\newcommand{\blambda}{{\utwi{\lambda}}}

\newcommand{\bphi}{{\utwi{\phi}}}

\newcommand{\bxi}{{\utwi{\xi}}}

\newcommand{\bSigma}{{\utwi{\Sigma}}}

\newcommand{\reals}{\mathbb{R}}

\newcommand{\sfT}{\textsf{T}}

\newcommand{\proj}{\mathsf{Proj}}

\newtheorem{lemma}{Lemma}
\newtheorem{theorem}{Theorem}

\theoremstyle{definition}

\newtheorem{assumption}{Assumption}
\theoremstyle{definition}
\newtheorem{remark}{Remark}

\usepackage{epstopdf}

\newcommand{\andrey}[1]{\textcolor{blue}{(Andrey says:  #1)}}
\newcommand{\yue}[1]{\textcolor{violet}{(Yue says:  #1)}}

\def\sfT{{\hbox{\tiny \textsf{T}}}}
\def\eqdef{\mathbin{:=}}

\newcommand\TTop{\rule{0pt}{2.6ex}}       
\newcommand\BBot{\rule[-1.2ex]{0pt}{0pt}} 

\newcommand{\probe}{\xi}
\newcommand{\bfprobe}{\bxi}

\newlength{\noteWidth}
\long\def\notes#1{\ifinner
             {\tiny #1}
             \else
              \marginpar{\parbox[t]{\noteWidth}{\raggedright\tiny #1}}
               \fi}

\newcommand{\shortoverline}[1]{\overline{#1\mkern1mu}\mkern2mu }
\newcommand{\shortunderline}[1]{\underline{#1\mkern-3mu}\mkern4mu }

\begin{document}

\title{\LARGE \bf 
Model-Free Primal-Dual Methods for Network Optimization with Application to Real-Time Optimal Power Flow}

\author{Yue Chen$^{1}$, Andrey Bernstein$^{1}$, Adithya Devraj$^{2}$, Sean Meyn$^{2}$%
	\thanks{*This work was authored (in part) by the National Renewable Energy Laboratory, operated by Alliance for Sustainable Energy, LLC, for the U.S. Department of Energy (DOE) under Contract No. DE-AC36-08GO28308.}
	\thanks{$^{1}$Y. Chen and A. Bernstein are with the Power Systems Engineering Center, National Renewable Energy Laboratory, Golden, CO 80401, USA
		{\tt\small yue.chen@nrel.gov, andrey.bernstein@nrel.gov}}%
	\thanks{$^{2}$A. Devraj and S. Meyn are with the Department of Electrical and Computer Engineering, University of Florida,       Gainesville, FL 32611, USA
        {\tt\small adithyamdevraj@ufl.edu, meyn@ece.ufl.edu}}%
}

\maketitle

\begin{abstract}
    This paper examines the problem of real-time optimization of networked systems 
    and develops online algorithms that steer the system towards the optimal trajectory without explicit knowledge of the system model. 
    The problem is modeled as a dynamic optimization problem with time-varying performance objectives and engineering constraints.
    The design of the algorithms leverages the online zero-order primal-dual projected-gradient method. In particular, the primal step that involves the gradient of the objective function (and hence requires networked systems model) is replaced by its zero-order approximation with two function evaluations using a deterministic perturbation signal. The evaluations are performed using the measurements of the system output, hence giving rise to a feedback interconnection, with the optimization algorithm serving as a feedback controller. The paper provides some insights on the stability and tracking properties of this interconnection. Finally, the paper applies this methodology to  a real-time optimal power flow problem in power systems, and shows its efficacy on the IEEE 37-node distribution test feeder for reference power tracking and voltage regulation.  
\end{abstract}


\IEEEpeerreviewmaketitle

\section{Introduction}

This paper considers the problem of real-time optimization of networked systems, where the desired operation of the network is formulated as a \emph{dynamic} optimization problem with time-varying performance objectives and engineering constraints \cite{Rahili2015,Fazlyab2016a,bernstein2019}. To illustrate the idea, consider the following time-varying model that describes the input-output behaviour of the  systems in the network:
\begin{align}
\label{eq:model}
\by(t) = \bh_t(\bx(t))
\end{align}
where $\bx(t) \in \mathbb{R}^n$ is a vector representing controllable inputs; $\by(t) \in \mathbb{R}^m$ collects the outputs of the network; and $\bh_t(\cdot): \mathbb{R}^n \rightarrow \mathbb{R}^m$ is a time-varying map representing the system model. Suppose   that the desired behaviour of the system is defined via a time-varying optimization problem of the form
\begin{align} 
\label{eq:model_opt}
\min_{\bx \in \cX(t), \by = \bh_t(\bx) } f_t(\by)
\end{align}
where $\cX(t)$ is a convex set representing, e.g., engineering constraints; and $f_t:  \mathbb{R}^{m}  \rightarrow \mathbb{R}$ is a convex function representing the time-varying performance goals.  For simplicity of illustration, we assume that the performance is measured only with respect to the output $\by$.

In an ideal situation, when infinite computation and communication capabilities are available, one could seek solutions (or stationary points) of \eqref{eq:model_opt} at each time $t$ by, e.g., an iterative projected-gradient method:
\begin{align}
\bx^{(k+1)} = & \proj_{\cX(t)}\Big\{ \bx^{(k)} - \alpha (\bJ_t^{(k)})^\sfT \nabla_{\by} f_t (h_t(\bx^{(k)}))  \Big\}, \label{eq:offline}
\end{align}
where $k = 0, 1, \ldots$ is the iteration index; $\bJ_t^{(k)} := \bJ_{\bh_t} (\bx^{(k)})$ is the Jacobian matrix of the input-output map $\bh_t$; $\proj_{\cX}(\bz) := \arg \min_{\bx \in \cX} \|\bz - \bx\|_2$ denotes projection; and $\alpha > 0$ is the step size. In practice, this approach might be infeasible. First, running \eqref{eq:offline} to convergence might be impossible due to stringent real-time constraints. Second, \eqref{eq:offline} requires the functional knowledge of the input-output map $\bh_t$ or an approximation thereof in real time. The latter is absent in typical network optimization applications, such as the power distribution system example considered here.

To address the first challenge, we adopt the measurement-based optimization framework of \cite{opfPursuit,Hauswirth2016,Hauswirth2018,bernstein2019,colombino2019}, wherein \eqref{eq:offline} is replaced with a running (or online) version:
\begin{align}
\bx^{(k+1)} = & \proj_{\cX^{(k)}}\Big\{ \bx^{(k)} - \alpha (\bJ^{(k)})^\sfT \nabla_{\by} f^{(k)}(\hat{\by}^{(k)})  \Big\}, \label{eq:online}
\end{align}
in which the iteration index $k$ is now identified with a discrete time instance $t_k$ at which the computation is performed; $\cX^{(k)} := \cX(t_k)$, $\bJ^{(k)} := \bJ_{\bh_{t_k}} (\bx^{(k)}) $, $f^{(k)} := f_{t_k}$; and $\hat{\by}^{(k)}$ represents the measurement of the network output $\bh_{t_k}(\bx^{(k)})$ at time $t_k$.  Note that \eqref{eq:online} represents a feedback controller that computes the input to the system based on its output; however, the implementation of this controller requires the real-time model information in the form of  $\bJ^{(k)}$. 

This paper sets out to address the second challenge, and tackle situations in which  $\bJ^{(k)}$ is not available for real time optimization. This is motivated by networked systems optimization problems, such as the optimal power flow in power systems considered in 
Section~\ref{sec:OPF}, wherein the topology and parameters of the network, as well as exogenous inputs, might change rapidly due to variability of working conditions, natural disasters, and cyber-physical attacks.   To address this challenge, we leverage a \emph{zero-order approximation} of the gradient of the objective function.  In particular, we replace the gradient of\footnote{A more direct approach is to estimate the Jacobian matrix with a zero-order approximation. This direction is a subject of ongoing research.} $F^{(k)} (\bx) := f^{(k)} (\bh_{t_k} (\bx))$
\begin{equation}
    \nabla F^{(k)} (\bx) = (\bJ_{\bh_{t_k}} (\bx))^{\sfT} \nabla_{\by} f^{(k)} (\bh_{t_k} (\bx))
\end{equation}
with the two function evaluation approximation:
\begin{equation} \label{eqn:proxy1}
    \hat{\nabla} F^{(k)} (\bx; \bfprobe, \epsy) \eqdef   \frac{1}{2\epsy} \bxi \left[ F^{(k)} (\bx + \epsy \bxi) - F^{(k)}(\bx - \epsy \bxi) \right]
\end{equation}
where $\bfprobe \in \reals^n$ is a perturbation (or exploration) vector,  and $\epsy > 0$ is a (small) scalar. Observe that this modification requires now measurements of the output at two (rather than one) inputs:  $\bx^{(k)} \pm \epsy \bfprobe$
at each time step $k$. With this modification, algorithm \eqref{eq:online}  can be carried out without any knowledge of the map $\bh_t$.

While the stylized algorithm \eqref{eq:online} and approximation \eqref{eqn:proxy1} is used here for an illustration of the main ideas, the paper considers a more general convex constrained optimization framework, and develops model-free \emph{primal-dual} methods to track the optimal trajectories. Using a deterministic exploration approach reminiscent to the quasi-stochastic approximation method \cite{berchecoldalmehmey19a}, this paper provides design principles for the exploration signal $\bfprobe$, as well as other algorithmic parameters, to ensure stability and tracking guarantees. In particular, we show that under some conditions, the iterates $\bx^{(k)}$ converge within a ball around the optimal solution of \eqref{eq:model_opt}.  
Finally, we apply the developed methodology to the problem of distributed real-time voltage regulation and power setpoint tracking in power distribution networks.

Much of the analysis is restricted to the case in which the model map $\bh_t$ is \emph{linear}.  Extensions to nonlinear models is a topic of current research.

\subsection*{Literature Review}
Zero-order (or gradient-free) optimization has been a subject of interest in the optimization, control, and machine learning communities for decades. The seminal paper of Kiefer and Wolfowitz  \cite{kiewol52} introduced a one-dimensional variant of  approximation \eqref{eqn:proxy1};  for a $d$-dimensional problem, it perturbs each dimension separately and requires $2d$ function evaluations.
The simultaneous perturbation stochastic approximation (SPSA) algorithm \cite{spa92}  uses zero-mean independent random perturbations,  
requiring two function evaluations at each step. 
In \cite{bhafumarwan03, prabhabhafumar19},  the SPSA algorithm was extended to \emph{deterministic perturbations}, to improve convergence rates 
under the assumption of a vanishing step-size \emph{and vanishing quasi-noise}.

Analysis of the zero-order methods based on two function evaluation (similar to \eqref{eqn:proxy1}) has been recently a focus of several papers \cite{duchi2015,Nesterov2017,Gao2018,hajinezhad2019}. 
In these papers,  a stochastic exploration signal $\bfprobe$ is used (typically, Gaussian iid) and convergence analysis of (projected) gradient methods and various variants of the method of multipliers is offered.

In the theoretical machine learning community, ``bandit optimization'' refers to zero-order online optimization algorithms based on a single or multiple evaluations of the objective function \cite{Flaxman2005,Awerbuch2008,Bubeck2012}. These algorithms typically have high variance of the estimate if the evaluation is performed using one or two function evaluations \cite{Chen2019}. Finally, the gradient-free technique termed  ``extremum-seeking control'' (ESC) \cite{arikrs03} leverages  sinusoidal perturbation signals and single function evaluation to estimate the gradient. Stability of the ESC feedback scheme was analyzed in, e.g., \cite{krswan00,wankrs00}. These algorithms typically require additional filters to smooth out the noise introduced by the perturbations.

In terms of zero-order network optimization, our work is closely related to the recent work \cite{hajinezhad2019}, wherein a multi-agent optimization of a class of non-convex problems is considered. The paper then develops distributed algorithms based on the zero-order approximation of the method of multipliers. In contrast to our paper, \cite{hajinezhad2019} considers a stochastic exploration signal  for the gradient estimation, and typically requires $N > 2$ function evaluations to reduce the estimation variance \cite[Lemma 1]{hajinezhad2019}; see also \cite{berchecoldalmehmey19a} for the detailed analysis of the advantage of deterministic vs stochastic exploration. Moreover, it considers a static optimization problem.

Contrary to much of the previous work (with the sole exception of perhaps \cite{hajinezhad2019}), we consider a constant step-size algorithm, necessitated by our application in a \emph{time-varying} optimization problem. Moreover, since the problem we consider here is a \emph{constrained} optimization problem, it leads to the extension of the gradient-free methods to novel model-free primal-dual algorithms. Due to the real-time nature of our application, we resort here to computationally-light online gradient-descent methods rather than variants of methods of multipliers which typically require solving an optimization problem at each time step.

\paragraph*{Organization}
The rest of the paper is organized as follows. Section \ref{sec:Network} introduces the general network optimization problem. Section \ref{sec:pri-dual} proposes the model-free primal-dual algorithm to solve the problem in Section \ref{sec:Network}. The convergence of the algorithm is analyzed in Section \ref{sec:analysis}. Section \ref{sec:OPF} introduces the power systems application and the simulation results are presented in Section \ref{sec:sim}. Section \ref{sec:conclusion} concludes the paper and proposes future works.

\section{Network Optimization Framework}
\label{sec:Network}
Consider the general network optimization problem  \cite{bernstein2019}, wherein at each time step $k \in \mathbb{N}$, the goal is to optimize the operation of a network of $N$ systems:
\begin{subequations} 
\label{eqn:sampledProblem}
\begin{align} 
 &\min_{\substack{\bx \in \reals^n}}\hspace{.2cm} f_0^{(k)}(\by^{(k)}(\bx)) + \sum_{i = 1}^N f_i^{(k)}(\bx_i) \label{eq:obj_p0} \\
&\mathrm{subject\,to:~} \bx_i \in \cX_i^{(k)} , \, i = 1, \ldots, N \label{eqn:constr_Xsampl} \\
& \hspace{1.8cm} g_j^{(k)}(\by^{(k)}(\bx))  \leq 0 , \, j = 1, \ldots, M \label{eq:ineqconst}
 \end{align}
\end{subequations}
where $\cX_i^{(k)} \subset \mathbb{R}^{n_i}$, with $\sum_{i = 1}^N n_i = n$, are convex sets representing operational constraints on the control input $\bx_i$ of system $i$; and $\by^{(k)}(\bx)$ is  an algebraic representation of some observable outputs in the networked system. In this paper, much of the analysis is restricted to the linear case\footnote{However, the developed algorithms, as well as our numerical results in Section \ref{sec:sim},  are \emph{not} restricted to the linear case.}
\begin{equation} \label{eq:sys}
\by^{(k)}(\bx) := \bC \bx + \bD \bw^{(k)} \in \reals^m    
\end{equation}
 where $\bC \in \mathbb{R}^{m \times n}$ and $\bD \in \mathbb{R}^{m \times w}$ are given model parameters, and $\bw^{(k)} \in \mathbb{R}^{w}$ is a vector of time-varying exogenous (uncontrollable) inputs. 
For example, in the power systems case considered in Section \ref{sec:sim}, $\by^{(k)}(\bx)$ represents the linearized power-flow equations. 

Further, the convex function $f_0^{(k)}(\by^{(k)}(\bx)): \mathbb{R}^n \rightarrow \mathbb{R}$ represent the time-varying performance goals associated with the outputs $\by^{(k)}(\bx)$; whereas 
 $f_i^{(k)}(\bx_i): \mathbb{R}^{n_i}  \rightarrow \mathbb{R}$ is a convex function representing performance goals of system $i$. Finally,  $g_j^{(k)}(\by^{(k)}(\bx)):  \mathbb{R}^{m}  \rightarrow \mathbb{R}$ are  convex functions used to impose time-varying constraints on the output $\by^{(k)}(\bx)$. 

Observe that \eqref{eqn:sampledProblem} is naturally \emph{model-based} -- in order to solve it, the knowledge of the (linearized) system model \eqref{eq:sys} and the exogenous input $\bw^{(k)}$ is required. This problem will be used next to define the \emph{desired operational trajectories} for the networked system; in Section \ref{sec:pri-dual}, we propose a \emph{model-free} algorithm to track these desired operational trajectories.

We define the following shorthand notation:
\[
\begin{aligned}
\bg^{(k)}(\by^{(k)}(\bx))  & := \left[g_1^{(k)}(\by^{(k)}(\bx)), \ldots, g_M^{(k)}(\by^{(k)}(\bx))\right]^\sfT
\\[.5em]
f^{(k)}(\bx) & := \sum_i f_i^{(k)}(\bx_i) 
\\
h^{(k)} (\bx) & := f^{(k)}(\bx) + f_0^{(k)}(\by^{(k)} (\bx)). 
\end{aligned} 
\]
Setting $\blambda \in \mathbb{R}_+^M$ as the vector of dual variables associated with~\eqref{eq:ineqconst},  the  Lagrangian function at step $k$ is given by: 
\begin{align}
\cL^{(k)}(\bx, \blambda) & :=h^{(k)} (\bx) + \blambda^\sfT\bg^{(k)}(\by^{(k)}(\bx)) \, . \label{eqn:lagrangian}
\end{align}
Finally, the \emph{regularized} Lagrangian function  is given by \cite{Koshal11}:
\begin{align} \label{eq:lagrangian_reg}
\cL^{(k)}_{p,d}(\bx, \blambda) := \cL^{(k)}(\bx, \blambda) + \frac{p}{2}\|\bx\|_2^2 - \frac{d}{2}\|\blambda\|_2^2 
\end{align}
where $p > 0$ and $d > 0$ are  regularization parameters. With this notation at hand, consider the saddle-point problem associated with $\cL^{(k)}_{p,d}$:
\begin{align} \label{eq:minmax}
\max_{{\footnotesize \blambda} \in \cD^{(k)}} \min_{\bx \in \cX^{(k)}} \cL^{(k)}_{p,d}(\bx, \blambda) \, \hspace{.5cm} k \in \mathbb{N}
\end{align}     
where $\cX^{(k)} := \cX_1^{(k)} \times \ldots \times \cX_N^{(k)}$ and $\cD^{(k)}$ is a convex and compact set that contains the optimal dual variable of \eqref{eq:minmax}; see, e.g.,  \cite{bernstein2019} for details. Hereafter, the optimal trajectory $\bz^{(*, k)} := \{\bx^{(*,k)}, \blambda^{(*,k)}\}_{k\in \mathbb{N}}$ of \eqref{eq:minmax} represents the \emph{desired trajectory}. 
It is clear that  $\cL^{(k)}_{p,d}(\bx, \blambda)$ is strongly convex in $\bx$ and strongly concave in $\blambda$; thus, the  optimizer $\bz^{(*, k)}$ of~\eqref{eq:minmax}  is unique. 
However, $\bz^{(*, k)}$ might be different from any saddle point of the exact Lagrangian $\cL^{(k)}(\bx, \blambda)$, and the bound on the distance between $\bz^{(*,k)}$ and the solution  of~\eqref{eqn:sampledProblem} can be established as, e.g., in \cite{Koshal11}.

To proceed, we list the assumptions that will be imposed on the different elements of \eqref{eqn:sampledProblem}.

\begin{assumption} 
\label{ass:constqualification}
Slater's condition holds at each time step $k$. 
\end{assumption}

\begin{assumption} 
\label{ass:cost}
The functions $f_0^{(k)}(\by)$ and $f^{(k)}_i(\bx_i)$ are convex and continuously differentiable. The  maps $\nabla f^{(k)}(\bx)$, $\nabla f_0^{(k)}(\by)$, and $\nabla^2 f_0^{(k)}(\by)$  are Lipschitz continuous.
\end{assumption}
\begin{assumption} 
\label{ass:nonlinearconstr}
For each $j = 1, \ldots, M$, the function $g_j^{(k)}(\by)$ is convex and continuously differentiable.  Moreover, $\nabla g_j^{(k)}(\by)$ and $\nabla^2 g_j^{(k)}(\by)$  are Lipschitz continuous. 
\end{assumption}

\begin{assumption} \label{ass:grad_var}
There exists a constant $e_f$ such that for all $k$ and $\bx, \by$:
\begin{align}
    \|\nabla f_0^{(k)}(\by) - \nabla f_0^{(k-1)}(\by)\| &\leq e_f, \nonumber \\
    \|\nabla f^{(k)}(\bx) - \nabla f^{(k-1)}(\bx)\|  &\leq e_f,   \nonumber \\
    \| \nabla g_j^{(k)}(\by) - \nabla g_j^{(k-1)}(\by)\| &\leq e_f, \quad j = 1, \ldots, M. \nonumber
\end{align}
\end{assumption}

\section{Model-Free Primal-Dual Method}
\label{sec:pri-dual}

A measurement-based primal-dual method was introduced in \cite{bernstein2019}  to track the desired trajectory $\{\bz^{(*, k)}\}$; we present the method here for convenience:
\begin{subequations} 
\label{eq:primalDualFeedback}
\begin{align}
\bx^{(k+1)} & = \proj_{\cX^{(k)}}\Big\{(1 - \alpha p) \bx^{(k)} - \alpha \big(\nabla_\bx f^{(k)}(\bx^{(k)}) \nonumber \\  
& \hspace{-.7cm} + \bC^\sfT\nabla_\by f_0^{(k)}(\hat{\by}^{(k)} ) 
 + \sum_{j = 1}^M \lambda_j^{(k)}\bC^\sfT\nabla g_j^{(k)}(\hat{\by}^{(k)} )
\big)\Big\}  \label{eq:primalstep} \\
\blambda^{(k+1)} & = \proj_{\cD^{(k)}}\Big\{(1 - \alpha d) \blambda^{(k)} + \alpha \bg^{(k)}(\hat{\by}^{(k)} ) \Big\} \label{eq:dualstep}
\end{align}
\end{subequations}
where $\alpha > 0$ is a constant step size, and $\hat{\by}^{(k)}$ is a (possibly noisy) measurement of $\by^{(k)}(\bx^{(k)})$ collected at time step $k$ (see Assumption \ref{ass:error} below).
 The main advantage of \eqref{eq:primalDualFeedback} is that it avoids explicit computation of \eqref{eq:sys} at each time step, and thus does not require knowledge or estimation of the uncontrollable input $\bw^{(k)}$ appearing in \eqref{eq:sys}. Note  that since the method is gradient-based, it relies on the availability of the network model matrix $\bC$. However, in many practical applications, including the power system example of Section \ref{sec:sim}, $\bC$ is time-varying (e.g., due to network topology variation) and its exact value is unavailable.

We next propose a model-free variant of \eqref{eq:primalDualFeedback}. 
As in the celebrated gradient-free stochastic approximation algorithm of  Kiefer and Wolfowitz \cite{kiewol52}, 
the general idea is to use few (online) evaluations of the objective function to estimate the gradient.     
The gradient approximation is justified through a Taylor series expansion as explained next.

Let  $F \colon \reals^n \rightarrow \reals$  be  a $C^3$ function,   whose first and second derivatives are globally Lipschitz continuous (cf.~Assumption \ref{ass:cost}).    For given $\bx\in\reals^n$,  this variable is  perturbed by $\pm\epsy \bfprobe$, with   $\bfprobe\in\reals^n$ and $\epsy > 0$. 
 According to Taylor's theorem,    
 for any $\bx\in\reals^n$ and $r\in \reals$, 
\[
F(\bx +r \bfprobe)  
= F(\bx) + r \bfprobe^{\sfT} \nabla F(\bx) + \frac{r^2}{2} \bfprobe^\sfT \nabla^2 F(\bx) \bfprobe 
+ O(r^3)
\]
Taking $r=\epsy$,  $r=-\epsy$, and then subtracting  yields:
 \begin{equation}
    F(\bx + \epsy \bfprobe) - F(\bx - \epsy \bfprobe) = 2 \epsy \bfprobe^{\sfT} \nabla F(\bx) + O(\epsy^3)  
 \end{equation}
 
 This approximation is summarized in Lemma~\ref{lem:twoPoint} that follows,  with the following shorthand notation:
\begin{equation} 
    \hat{\nabla} F (\bx; \bfprobe, \epsy) :=   \frac{1}{2\epsy} \bxi \left[ F(\bx + \epsy \bxi) - F(\bx - \epsy \bxi) \right].
\end{equation}
\begin{lemma} 
\label{lem:twoPoint}
Let $F: \reals^n \rightarrow \reals$ be a $C^3$   function,   with Lipschitz continuous $\nabla F$ and $\nabla^2 F$. 
Then,
\begin{equation}
\label{eqn:proxy}
    \hat{\nabla} F (\bx; \bfprobe, \epsy)  =   \bxi \bxi^{\sfT} \, \nabla F(\bx) + O(\epsy^2) 
\end{equation} 
\end{lemma}

Note that $O(\epsy^2) = \mathbf{0}$ if $F$ is a quadratic function.

In applying Lemma~\ref{lem:twoPoint} to algorithm design,  the   vector $\bfprobe$ is   updated at each iteration of the algorithm.      
Let $\{\bfprobe^{(k)}\}$ denote a deterministic \emph{exploration signal}, with $\bfprobe^{(k)} \in \reals^n$.  
It is assumed in this paper that the exploration signal is obtained  by sampling  the sinusoidal signal
\begin{equation} \label{eq:probe}
\probe_i(t) = \sqrt{2}\sin (\omega_i t), \quad i = 1, \ldots, n,   
\end{equation}
with $\omega_i \neq \omega_j$ for all $i \neq j$; 
 other signals are possible provided that they satisfy Assumption \ref{a:pertb} below.

The following gradient-free variant of the primal-dual algorithm \eqref{eq:primalDualFeedback} is well  motivated by 
Lemma~\ref{lem:twoPoint}.

\noindent\makebox[\linewidth]{\rule{\columnwidth}{0.4pt}}
\textbf{Model-Free Primal-Dual Algorithm} \\
\noindent\makebox[\linewidth]{\rule{\columnwidth}{0.4pt}}

\noindent At each time step $k$, perform the following steps:

\noindent  \textbf{[S1a] (forward exploration)}: Apply $\bx^{(k)}_+ := \bx^{(k)} + \epsy \bfprobe^{(k)}$ to the system, and collect the  measurement $\hat{\by}^{(k)}_+$ of the output $\by^{(k)}(\bx^{(k)}_+)$.

\noindent  \textbf{[S1b] (backward exploration)}: Apply $\bx^{(k)}_- := \bx^{(k)} - \epsy \bfprobe^{(k)}$ to the system, and collect the  measurement $\hat{\by}^{(k)}_-$ of the output $\by^{(k)}(\bx^{(k)}_-)$.

\noindent  \textbf{[S1c] (control application)}: Apply $\bx^{(k)}$ to the system, and collect the  measurement $\hat{\by}^{(k)}$ of the output $\by^{(k)}(\bx^{(k)})$.

\noindent  \textbf{[S2a] (approximate gradient)}: Compute the approximate gradient for the primal step
\begin{align}
    \hat{\nabla} \cL^{(k)} &:= \nabla_\bx f^{(k)}(\bx^{(k)}) \nonumber\\
    &\quad + \frac{1}{2 \epsy}\bfprobe^{(k)}\left[f_0^{(k)}(\hat{\by}^{(k)}_+ ) -f_0^{(k)}(\hat{\by}^{(k)}_- )  \right]\nonumber\\
    &\quad + \frac{1}{2 \epsy} \bfprobe^{(k)}(\blambda^{(k)})^\sfT \left[\bg^{(k)}(\hat{\by}^{(k)}_+ ) - \bg^{(k)}(\hat{\by}^{(k)}_- )  \right].
\end{align}

\noindent  \textbf{[S2b] (approximate primal step)}: Compute
\begin{equation} \label{eqn:primal}
\bx^{(k+1)} = \proj_{\cX^{(k)}}\left\{(1 - \alpha p) \bx^{(k)} - \alpha \hat{\nabla} \cL^{(k)} \right\}.
\end{equation}

\noindent  \textbf{[S3] (dual step)}: Compute
\begin{equation} \label{eqn:dual}
    \blambda^{(k+1)}  = \proj_{\cD^{(k)}}\left\{(1 - \alpha d) \blambda^{(k)} + \alpha \bg^{(k)}(\hat{\by}^{(k)} ) \right\}.
\end{equation}

\noindent\makebox[\linewidth]{\rule{\columnwidth}{0.4pt}}

The following remarks are in order. 

\begin{remark}
  Lemma~\ref{lem:twoPoint} alone is not enough to obtain convergence guarantees since the matrix $\bxi \bxi^{\sfT}$ is  semi-definite but not strictly definite. We formulate conditions for convergence in the next section.
\end{remark}

\begin{remark}
The algorithm allows distributed implementation.     This point is illustrated in a  power system optimization example described in Section~\ref{sec:sim}.
\end{remark}

\begin{remark}
The algorithm requires three measurements of the output: $\hat{\by}^{(k)}_+, \hat{\by}^{(k)}_-,$ and $\hat{\by}^{(k)}$. In fact, the third measurement can be replaced with, e.g., an average $\hat{\by}^{(k)} := 0.5(\hat{\by}^{(k)}_+  + \hat{\by}^{(k)}_-)$, provided that $\hat{\by}^{(k)}$ satisfies Assumption \ref{ass:error} below.
\end{remark}

\section{Convergence Analysis} \label{sec:analysis}
For the purpose of the analysis, we assume bounded measurement error.

\begin{assumption} \label{ass:error}
There exists a scalar $e_y < \infty$ such that the measurement error can be bounded as  
\begin{align*}
\sup_{k \geq 1} \|\hat{\by}^{(k)} - \by^{(k)}(\bx^{(k)})\|_2 & \leq e_y, \nonumber \\
\sup_{k \geq 1} \|\hat{\by}_+^{(k)} - \by^{(k)}(\bx_+^{(k)})\|_2 & \leq e_y, \\
\sup_{k \geq 1} \|\hat{\by}_-^{(k)} - \by^{(k)}(\bx_-^{(k)})\|_2 & \leq e_y. \nonumber
\end{align*}
\end{assumption}

The exploration signal plays an important role in the convergence of the proposed algorithm. Recall that the exploration signal $\bfprobe^{(k)}$ is sampled from a continuous-time signal $\bfprobe(t)$. 
Recent work on quasi-stochastic approximation \cite{shimey2011,berchecoldalmehmey19a} and  extremum seeking control \cite{ariyur2003real} has shown that certain deterministic exploration signals can help reduce the asymptotic variance and improve convergence, in comparison to random exploration (see Fig.~1 of \cite{shimey2011}). We impose the following assumption on the exploration signal.

\begin{assumption}
\label{a:pertb}
The exploration signal is a periodic signal with period $T$, and 
\begin{eqnarray} 
\frac{1}{T} \int_{t }^{t+T} \bfprobe(\tau) \bfprobe(\tau)^\sfT \, d\tau &=& \bI
\label{eq:unit_cov}
\end{eqnarray}
for all $t$, where $\bI$ is the identity matrix.
\end{assumption}
It can be shown that the signal defined in \eqref{eq:probe} satisfies Assumption \ref{a:pertb} with $T$ being a common integer multiple of the sinusoidal signal periods.

This strong assumption is imposed to simplify the proofs that follows.  It is enough to have the approximation:
\[
\frac{1}{T} \int_{t }^{t+T} \bfprobe(\tau) \bfprobe(\tau)^\sfT \, d\tau  =  \bSigma_\xi   +  O(1/T)
\]
in which $ \bSigma_\xi$ is positive definite,    and the $1/T$ bound for the error term is independent of $t$ (this is the assumption used in \cite{berchecoldalmehmey19a}). 
An example is the signal defined in \eqref{eq:probe}  in which the frequencies may not have a common multiple.   

We first show that the  primal step \eqref{eqn:primal} is approximately equivalent to an averaged primal step, where the singular ``gain matrix'' $\bfprobe \bfprobe^{\sfT}$ is replaced with the identity matrix \eqref{eq:unit_cov}. This analysis  is performed for the algorithm defined in continuous time,  which is justified using standard ODE approximation techniques from the stochastic approximation literature \cite{bor08a,berchecoldalmehmey19a} or more recent literature on optimization \cite{zhamoksrajad18,shidusujor19}.
We then apply the results of \cite{bernstein2019} to the resulting approximate primal-dual algorithm to show tracking of the desired trajectory $\{\bz^{(*, k)}\}$ defined by \eqref{eq:minmax}.

\subsection{Averaged Primal Step}
In this section, we provide some preliminary results on the asymptotic properties of the primal iteration \eqref{eqn:primal}; a detailed analysis is the subject of ongoing work. 

To that end, consider a $C^3$ strongly convex strongly smooth function $F$ and the associated projected gradient descent method:
\begin{equation}
    \bx^{(k+1)} = \proj_{\cX} \left \{\bx^{(k)} - \alpha \nabla F(\bx^{(k)}) \right \}.
\end{equation}
The corresponding gradient-free method reads
\begin{equation} \label{eqn:gf}
    \bx^{(k+1)} = \proj_{\cX} \left \{\bx^{(k)} - \alpha \hat{\nabla} F(\bx^{(k)}; \bfprobe^{(k)}, \epsy)\right \}.
\end{equation}
Our goal in this section is to explain conditions for convergence of \eqref{eqn:gf}. 

We cannot expect convergence of $\{\bx^{(k)} \}$ to a limit as $k\to\infty$.     We might expect something like the ergodic steady-state obtained for stochastic approximation with fixed stepsize~\cite{bor08a}.   In this section, we are content to simply obtain bounds on the error $\| \bx^{(k)} - \bx^* \|$,   where $\bx^*$ is the minimizer of $F$.
Under appropriate assumptions, for a continuous time model we obtain
\begin{equation}
\limsup_{t\to\infty}  \| \bx(t) -\bx^* \| = O(\alpha+ \epsy^2)
\label{e:ultBdd}
\end{equation}
See Remark~\ref{r:Jbdd} at the close of this section.

According to Lemma~\ref{lem:twoPoint}, the continuous-time analogue of \eqref{eqn:gf} is given by
\begin{equation} 
\label{eqn:gf_cont_proj}
    \dot{\bx}(t) = \proj_{\cT_{\cX}(\bx(t))} \left \{     \alpha  \bbeta(t, \bx(t)) \right\}
\end{equation} 
where,   $\cT_{\cX}(\bx)$ is the cone of tangent directions of $\cX$ at $\bx$ (or for short, the tangent cone of $\cX$ at $\bx$), and
  for any $t$ and $ \bx$, 
\begin{equation}
    \bbeta(t, \bx) :=   - \bxi(t) \bxi(t)^{\sfT} \, \nabla F(\bx) + O(\epsy^2).
\end{equation}
The function $\bbeta$ is Lipschitz continuous in $\bx$, uniformly in $t$.   


Analysis of the continuous projection in \eqref{eqn:gf_cont_proj} is a topic of research.   We consider here a simplification in which the projection is only applied at integer multiples of the period $T$.   At stage $K+1$ of the algorithm, we have computed $ \bx(KT)$,  and   with this initial condition, $\{ \bx(t) :  KT\le t  < (K+1) T \}$ is defined as the solution to the ODE without projection:
 \begin{equation} 
    \dot{\bx}(t)   =     \alpha  \bbeta(t, \bx(t))   
\label{eqn:gf_cont} 
\end{equation}
We then define $   \bx((K+1)T)           = $
 \begin{equation} 
 \proj_K \Bigl\{  \bx(KT) +   \alpha \int_{KT}^{(K+1)T}    \bbeta(\tau , \bx(\tau))\,  d\tau   \Bigr\}
\label{eqn:gf_cont_BigStep} 
\end{equation}
where the projection is onto the set of states corresponding to constraints on   $ \bx(KT)$.  

This recursion admits an approximation by a projected gradient descent algorithm:

\begin{lemma} \label{lem:equiv_dyn}
Suppose $F$ is $C^3$,  and that both $\nabla F$ and $\nabla^2 F$ are globally Lipschitz continuous.   Suppose moreover that   
Assumption~\ref{a:pertb} holds.   
%
%
Then,  the solutions to  \eqref{eqn:gf_cont}  admit the approximation:
\begin{equation}
   \bx((K+1)T)    =\proj_K \left \{\bx  - \alpha     T (\nabla F(\bx) +   s_K(\bx))   \right \}
\end{equation}
where    $\bx=\bx(KT)$,   and $s_K(\bx) = O(  \alpha+  \epsy^2 )$.
\end{lemma}

\begin{proof}
It is enough to approximate the integral appearing in \eqref{eqn:gf_cont_BigStep}.   Under the Lipschitz conditions and boundedness of $\bfprobe$,  
there exists a constant $b_0<\infty$ such that for all $KT\le \tau  < (K+1) T$, 
\[
\begin{aligned}
\| \bx(\tau) - \bx(KT) \|  &  \le b_0 T\alpha, 
   \\
\| \bbeta(\tau, \bx(\tau) ) -   \bbeta(\tau, \bx(KT) ) \|  &  \le b_0 T\alpha. 
\end{aligned} 
\]
Consequently,  for   $\tau$ in this range,
\[
\begin{aligned}
 \bbeta(\tau,\bx(\tau))  &=    \bbeta(\tau, \bx(KT) )  + O(\alpha )
   \\
			 &=    
			  - \bxi(\tau) \bxi(\tau)^{\sfT} \, \nabla F(\bx(KT)) + O(\epsy^2)   + O(\alpha )
\end{aligned} 
\] 
Under Assumption \ref{a:pertb}, it follows that
\[
 \int_{KT}^{(K+1)T}\bbeta (\tau, \bx(\tau)) d\tau = -  T \left[\nabla F(\bx(KT)) + O(\alpha+ \epsy^2) \right] \,,
\]
which completes the proof.
\end{proof}

\begin{remark}
\label{r:Jbdd}
If $\| \nabla F \|$ is coercive,  then techniques used to establish stability of gradient descent can be used to show that the solution to \eqref{eqn:gf_cont}  is ultimately bounded \cite{nes14}.  If $F$ is $L$-smooth and $\mu$-strongly convex then we   obtain the uniform bound \eqref{e:ultBdd}, independent of the initial condition.
\end{remark}

\subsection{Tracking of the Desired Trajectory}

We note that Lemma~\ref{lem:equiv_dyn} suggests that the primal iteration \eqref{eqn:primal} can be viewed as an approximate projected gradient descent iteration on the primal variable provided that: (i) $\alpha$ and $ \epsilon$ are chosen small enough, and (ii) the projection is performed periodically, every $T$ time instances.
In that case, iteration \eqref{eqn:primal}-\eqref{eqn:dual} is  equivalent to:
\begin{subequations} \label{eqn:equiv_iter}
\begin{align} 
\bx^{(k+1)} &= \proj_k \Big\{ \bx^{(k)} - \alpha \Big(\nabla_{\bx}  \cL_{p,d}^{(k)}(\bz^{(k)})  \nonumber\\
&\qquad \qquad \qquad \quad + O(\alpha + \epsilon^2 + e_f + e_y) \Big) \Big\} \\
\blambda^{(k+1)}  &= \proj_{\cD^{(k)}}\left\{(1 - \alpha d) \blambda^{(k)} + \alpha \bg^{(k)}(\hat{\by}^{(k)} ) \right\},
 \end{align}
\end{subequations}
where where $\bz^{(k)} := \{\bx^{(k)}, \blambda^{(k)}\}$, and $\proj_k$ is the projection  onto $\cX^{(k)}$ for $k$ a multiple of $T/\Delta t$ ($\Delta t$ is the time discretization step).
We note that the additional term $O(e_y)$ in \eqref{eqn:equiv_iter} is due to the use of measurements in \eqref{eqn:primal} instead of the exact system output (cf.~Assumption \ref{ass:error}). Finally, the additional term $O(e_f)$ is due to the time-variability of the Lagrangian over the period of $T$ time instances as quantified by Assumption \ref{ass:grad_var}. 

We next analyze the tracking properties of \eqref{eqn:equiv_iter}. To this end,  quantify the temporal variability of the desired trajectory $\bz^{(*, k)}$ via  
\begin{equation} \label{eqn:sigma2}
\sigma^{(k)} := \|\bz^{(*, k+1)} - \bz^{(*, k)}\|_2 \, .
\end{equation}
In addition, let
\begin{equation}
\label{phi_mapping_true}
\bphi^{(k)}: \bz \mapsto 
 \left[\begin{array}{c}
 \nabla_{\bx}  \cL_{p,d}^{(k)}(\bz) \\
 - \bg^{(k)}(\by^{(k)}(\bx)) + d \blambda 
\end{array}
\right]
\end{equation}
denote the primal-dual operator associated with \eqref{eq:minmax}. The following holds (see, e.g., \cite[Lemma 5]{bernstein2019}).
\begin{lemma}
\label{lem:Phimonotone} There exist constants $0 < \eta_{\phi} \leq L_\phi < \infty$ such that, 
for every $k \in \mathbb{N}$, the map $\bphi^{(k)}(\bz)$ is strongly monotone over $\cX^{(k)} \times \cD^{(k)}$ with constant $\eta_{\phi}$ and Lipschitz over $\cX^{(k)} \times \cD^{(k)}$ with coefficient $L_\phi$.
\end{lemma}

Finally, let
\begin{equation}
\label{phi_mapping}
 \hat{\bphi}^{(k)}: \bz \mapsto 
 \left[\begin{array}{c}
 \nabla_{\bx} \cL_{p,d}^{(k)}(\bz) + O(\alpha + \epsilon^2 + e_f + e_y) \\
 - \bg^{(k)}(\hat{\by}^{(k)}) + d \blambda 
\end{array}
\right],
\end{equation}
denote the approximate primal-dual operator associated with \eqref{eqn:equiv_iter}. Using   \cite[Lemma 4]{bernstein2019}, the following result follows.

\begin{lemma}
There exists a constant $\epsy_\phi = O(
\alpha + \epsy^2 + e_f + e_y)$ such that
\begin{align}
\label{eq:errorPhi}
\left\|\bphi^{(k)}(\bz^{(k)}) - \hat{\bphi}^{(k)} (\bz^{(k)}) \right\|_2 \leq \epsy_\phi, \quad \forall k = 1, 2, \ldots
\end{align}
\end{lemma}

We use \cite[Theorem 4]{bernstein2019} to show that the approximate time-varying primal-dual algorithm \eqref{eqn:primal}--\eqref{eqn:dual} exhibits the following tracking performance.

\begin{theorem} \label{cor:asym}
Suppose that the step size $\alpha$ is chosen such that $0 < \alpha  < 2 \eta_\phi / L_\phi^{2}$. Moreover, suppose that there exists a scalar $\sigma <  \infty$ such that $\sup_{k \geq 1} \sigma^{(k)} \leq \sigma$. Then, the sequence  $\{\bz^{(k)}\}$ converges Q-linearly to $\{\bz^{(*,k)}\}$ up to an asymptotic error bound given by:
\begin{align} 
\limsup_{k\to\infty} \|\bz^{(k)} - \bz^{(*, k)}\|_2 &\leq \frac{\alpha  \epsy_\phi + \sigma}{1 - c}  \label{eqn:asym_bound_apr}
\end{align}
where
$c :=  \sqrt{1 - 2 \alpha  \eta_\phi + \alpha^2 L_\phi^2} < 1$
and 
    $\epsy_\phi =  O(\alpha + \epsy^2 + e_f + e_y).$

\end{theorem}

\section{Model-Free Real-Time Optimal Power Flow}

\label{sec:OPF}
The application considered in this paper is real-time optimization of the power injections of distributed energy resources (DERs) in a distribution system.  The following objectives are pursued: (i)  feeder head/substation power is following a prescribed signal, and (ii) the voltages are regulated within prescribed bounds. This is achieved via formulating an optimal power flow (OPF) problem of the form \eqref{eqn:sampledProblem} and seeking its solutions in real time using the model-free primal-dual framework introduced in Section \ref{sec:pri-dual}.


Consider a distribution system that is described by a set of nodes $\mathcal{N} = \{ 0, \dots, n\}$, where  node $0$ denotes the feeder head (or substation) with fixed voltage; the rest are $PQ$ nodes. 
In this work, we focus on controlling the batteries and PV systems in the distribution system interfaced via power inverters; however, the proposed framework  can be easily extended to include other types of DERs.  
Let   $N_{bt}$ and $N_{pv}$ be the number of batteries and PVs, respectively. 
So the total number of DERs $ N_{der} = N_{bt} + N_{pv}$. 

\subsection{Notation}
For notation brevity,  in this section, we drop the algorithm iteration superscript $^{(k)}$ from variable names, unless otherwise specified.

The distribution system is described by the AC power-flow equations, defining input-output relationship of this physical system.
In particular:
\begin{itemize}
\item $\bx \in \mathbb{R}^{2N_{der}}$ collects all decision variables, where each DER has real and reactive power  $\bx_i = \{\bx_{i,p}, \, \bx_{i,q}\}$;
\item $\bl \in \mathbb{R}^{2n}$ collects the active and reactive power of uncontrollable loads (the ``noise'');
\item $\bv(\bx, \bl) \in \mathbb{R}^n$ collects the voltage magnitudes at every $PQ$ node; and
\item $P_0(\bx, \bl) \in \mathbb{R}$ is the power flow at the feeder head.
\end{itemize}
The (non-linear) maps $(\bv(\bx, \bl), P_0(\bx, \bl))$ exists under some conditions and are defined by the power-flow equations \cite{bernstein2018}. 
The goal is to control the output $(\bv(\bx, \bl), P_0(\bx, \bl))$ by controlling $\bx$.

In order to formulate a \emph{convex} target OPF \eqref{eqn:sampledProblem}, an approximate linearized relationship between the output $\by := (\bv, P_0)$ and the input $(\bx, \bl)$ is leveraged. In particular, at each time step $k$, the output is approximated as
\begin{equation} \label{eqn:lin_pf}
    \by = \bC \bx + \bD \bl + \by_0;
\end{equation}
see, e.g., \cite{bernstein2018}. We would like to emphasize that the relationship \eqref{eqn:lin_pf} is used only to define the desired trajectory via \eqref{eqn:sampledProblem} and \eqref{eq:minmax}; the actual simulation below is performed using the exact (nonlinear) power-flow equations.

\subsection{Objective Function}
  The objective function \eqref{eq:obj_p0} consists of two parts:
\begin{subequations}\label{eq:obj0}
	\begin{align}
	f_0(\bx) = & ({P}_{0}(\bx, \bl) - {P}_0^\bullet)^2   \label{eq:obj_Ptrack}\\
		f_i(\bx_i) = & c_{i,p} \bx_{i,p}^2 + c_{i,q} \bx_{i,q}^2 + c_{i,p^\bullet} (\bx_{i,p} - \bx_{i,p} ^{\bullet})^2 
 \label{eq:obj_DER}
	\end{align}
\end{subequations} 
where $c_{i,p}$, $c_{i,q}$, and $c_{i,p^\bullet}$ are weighting parameters for the DER $i$.
The global cost \eqref{eq:obj_Ptrack} drives the substation real power $P_{0}$ to follow the reference power $P^\bullet_0$, which is usually time-varying. The local DER cost  \eqref{eq:obj_DER} penalizes the control effort and the difference between the control and its desired reference power. For a PV system, the reference power $\bx_{i,p}^\bullet$ is the real-time maximal PV power generation; for a battery system, the reference  $\bx_{i,p} ^{\bullet}$ is the power  to reach the preferred state of charge (SOC), which is usually set as the middle point of SOC to allow maximal operation flexibility.

\subsection{Constraints}
Constraints associated with the OPF are summarized below. 
\subsubsection{Voltage}
Each node needs to maintain a stable voltage magnitude within a certain range. For any node $i \in \mathcal{N}$,
\begin{equation}
	 \shortunderline{V}_i \leq \bv_i(\bx, \bl) \leq  \shortoverline{V}_i
	 \label{eq:constr_v}
\end{equation}

\subsubsection{Battery}
Battery limits are imposed on battery control variables to enable feasible operations. 
Each  battery has its power and energy limits:
\begin{subequations}\label{eq:constr_battery}
	\begin{align*}
	&  \shortunderline{X}_{i,p}  \leq \bx_{i,p} \leq  \shortoverline{X}_{i,p}, \quad \shortunderline{SOC}_i \leq SOC_i \leq \shortoverline{SOC}_i 
	\end{align*}
\end{subequations} 
where $SOC_i$ represents the SOC of battery $i$. 

Consider the battery charging/discharging dynamics 
\[
SOC_i^{(k+1)} = SOC_i^{(k)} + \bx_{i,p}^{(k)} \Delta t,
\]
where $\Delta t$ is the time step in one iteration.
The energy limit can be rewritten as power constraints
\[
\frac{\shortunderline{SOC}_i - SOC_i^{(k+1)}}{\Delta t} \leq \bx_{i,p}^{(k)} \leq \frac{\shortoverline{SOC}_i - SOC_i^{(k+1)}}{\Delta t} 
\]
Letting
\begin{subequations}
\begin{align}
\shortunderline{X}_i^{bt} &= \max\left(\shortunderline{X}_{i,p}\,, \, \frac{\shortunderline{SOC}_i - SOC_i^{(k+1)}}{\Delta t} \right) \nonumber \\
\shortoverline{X}_{i}^{bt} &= \min\left(\shortoverline{X}_{i,p}^{bt}\, , \, \frac{\shortoverline{SOC}_i - SOC_i^{(k+1)}}{\Delta t}\right) \nonumber
\end{align}
\end{subequations}
we obtain following constraints for battery $i$:
\begin{equation}\label{eq:constr_battery2}
	 \shortunderline{X}_i^{bt}  \leq \bx_{i,p} \leq  \shortoverline{X}_i^{bt}, \quad  \bx_{i,p}^2  + \bx_{i,q}^2 \leq  (\shortoverline{S}_i^{bt})^2 
\end{equation} 
where $\shortoverline{S}_i^{bt}$ is the upper bound of battery apparent power.

\subsubsection{PV} 
Driven by time-varying solar radiation, PV systems have time-varying limits. Let $\shortoverline{P}_i^{pv}$ and 
$\shortoverline{S}_i^{pv}$ denote the available activate power and apparent power, respectively. The following constraints are imposed on PV $i$:
\begin{equation}\label{eq:constr_PV}
	 0  \leq \bx_{i,p} \leq  \shortoverline{P}_i^{pv}, \quad   \bx_{i,p}^2  + \bx_{i,q}^2 \leq  (\shortoverline{S}_i^{pv})^2 
\end{equation}

\subsection{Distributed Algorithm Implementation}
The OPF associated with cost function \eqref{eq:obj0} and constraints \eqref{eq:constr_v}, \eqref{eq:constr_battery2},  \eqref{eq:constr_PV} can be reformulated as a saddle-point problem of the form of  \eqref{eq:minmax}, which can then be solved using the model-free primal-dual algorithm proposed in Section \ref{sec:pri-dual}. 

In fact, the Lagrangian \eqref{eq:lagrangian_reg} can be decomposed with respect to each DER.  The cost function \eqref{eq:obj_DER} and constraints \eqref{eq:constr_battery2},  \eqref{eq:constr_PV} are concerned with only local DER information. They have an explicit form of gradient with respect to each local DER control/dual variables. In addition, the measurement-based global information  $f_0(\hat{\by}_+)$ and  $f_0(\hat{\by}_-)$ that is associated with \eqref{eq:obj_Ptrack}, as well as the constraint function $\bg(\hat{\by})$ that is associated with \eqref{eq:constr_v}, can be broadcast to each DER in real time. Therefore, each DER has necessary information to locally update its control and dual variables. This gather-and-broadcast control architecture was previously considered in, e.g., \cite{opfPursuit,unified}, in the model-based context.

One practical concern in the distributed implementation is the asynchronicity in the control signals and measurements. At each time step, it is possible that DERs do not implement their control signals at the same time, and the measurements might not capture the system response after all local control implementations. It is particularly true for fast communication and control. In fact, the asynchronicity can be modeled as an additional noise source, and can be analyzed similarly to, e.g., \cite{fixedPoints}. In the next section, we numerically evaluate the sensitivity of our approach to different levels of noise.

\section{Numerical Study}

\label{sec:sim}
\subsection{Simulation Setup}

Numerical experiments were conducted on a modified IEEE 37-node test feeder to demonstrate the proposed model-free OPF. We used a single-phase variant of the original test feeder \cite{kersting2006distribution}, with DERs connected at different locations. Fig. \ref{fig: IEEE37bus} shows the feeder diagram. Table \ref{t:DER} provides the configuration for these DERs, and the efficiency of the batteries is assumed to be 0.9. 

\begin{figure}[t]
    \centering
    \includegraphics[width=.35\textwidth]{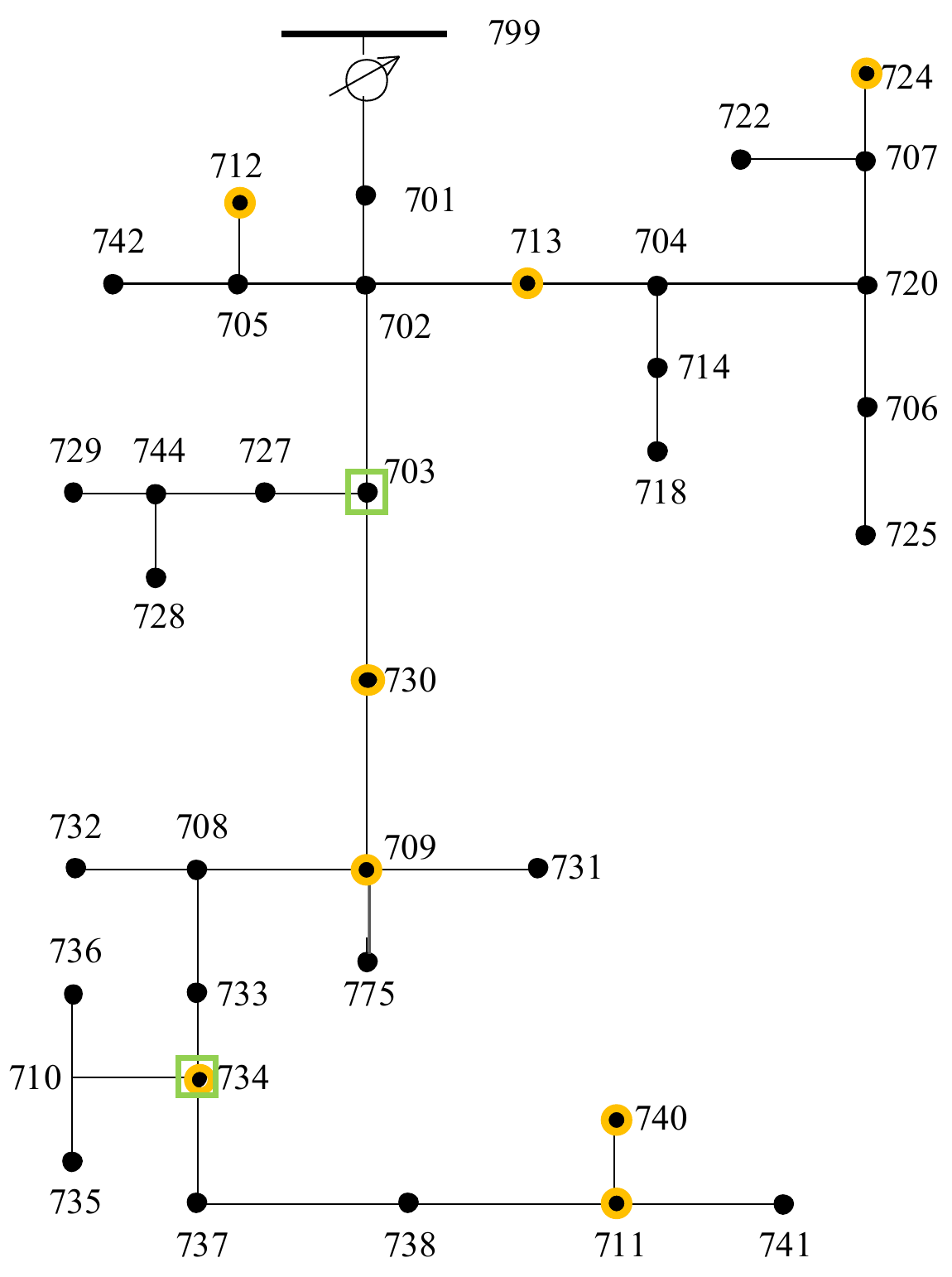}
    \caption{Modified IEEE 37-node test feeder: PV and battery nodes are labeled by circles and squares, respectively.}
    \label{fig: IEEE37bus}
\end{figure}

\begin{table}[]
\caption{Configuration of DERs}
\label{t:DER}
\begin{tabular}{|c|c|c|c|c|}
\hline
\multirow{2}{*}{\textbf{Node}}         & \multirow{2}{*}{\begin{tabular}[c]{@{}c@{}}\textbf{PV}\\  \textbf{Capacity}\\ (kVA)\end{tabular}} & \multicolumn{3}{c|}{\textbf{Battery Capacity}}                                                                                                                                                                                                                             \\ \cline{3-5} 
                              &                                                                              & \begin{tabular}[c]{@{}c@{}}\textbf{SOC}\\ (MWh)\end{tabular} & \begin{tabular}[c]{@{}c@{}}\textbf{Active Power} \\ (MW)\end{tabular} & \begin{tabular}[c]{@{}c@{}}\textbf{Apparent Power}\\ (MVA)\end{tabular} \\ \hline \hline
\textit{703}    & N/A                                                                          &  $[0 \quad 30] $                                                           &  $[-10 \quad 10] $                                                                       & 12                                                                                       \\ \hline
\textit{709}   & 200                                                                          & \multicolumn{3}{c|}{N/A}                                                                                                                                                                                                                                        \\ \hline
\textit{711}   & $200 $                                                                       & \multicolumn{3}{c|}{N/A}                                                                                                                                                                                                                                        \\ \hline
\textit{712} & $200 $                                                                       & \multicolumn{3}{c|}{N/A}                                                                                                                                                                                                                                        \\ \hline
\textit{713}  & $100 $                                                                       & \multicolumn{3}{c|}{N/A}                                                                                                                                                                                                                                        \\ \hline
\textit{724}  & $100$                                                                        & \multicolumn{3}{c|}{N/A}                                                                                                                                                                                                                                        \\ \hline
\textit{730}  & $200$                                                                        & \multicolumn{3}{c|}{N/A}                                                                                                                                                                                                                                        \\ \hline
\textit{734}  & $200$                                                                        & $[0 \quad 30] $                                                           & $[-10 \quad 10] $                                                                 & 12                                                                                       \\ \hline
\textit{740}   & $200$                                                                        & \multicolumn{3}{c|}{N/A}                                                                                                                                                                                                                                        \\ \hline
\end{tabular}
\end{table}

\begin{figure}[h]
    \centering
    \includegraphics[width=.4\textwidth]{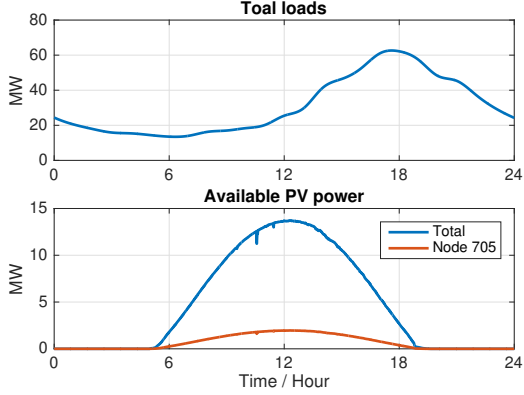}
    \caption{Total load and available PV generation in one day,  and an example of  available PV generation at one node.}
    \label{fig:loadpv}
\end{figure}

\begin{table*}[h]
\caption{Algorithm parameters used in the simulation.}
\begin{tabular}{|c|c|c|c|c|c|c|c|c|c|c|c|}
\hline
 \textbf{Coefficients in the cost function} \eqref{eq:obj_DER} &  $c_{i,p}^{pv}$ & $c_{i,q}^{pv}$ & $c^{pv}_{i,p^\bullet}$ & $c_{i,p}^{bt}$ &$c_{i,p}^{bt}$ & $c_{i,p^\bullet}^{bt}$    \TTop \BBot   \\ \hline 
\textbf{Value }   &  $1\times 10^{-5}$ & $1\times 10^{-5}$ & $1\times 10^{-3}$ & $1/6\times 10^{-4}$ &  $1/6\times 10^{-4}$ &  $1/6\times 10^{-4}$   \TTop \BBot  \\ \hline 
\end{tabular}

  \smallskip
  
\begin{tabular}{|c|c|c|c|c|c|c|c|c|c|c|c|}
\hline
\textbf{Step sizes (control variables)} & $\alpha^{pv}_{i,p} $ & $a^{pv}_{i,q} $ & $a^{bt}_{i,p} $ & $a^{bt}_{i,q} $ \TTop \BBot   \\ \hline
\textbf{Value }    & $2$ &  $2$ &  $12$ & $12$  \TTop \BBot  \\ \hline
\end{tabular}

  \smallskip
  
\begin{tabular}{|c|c|c|c|c|c|c|c|c|c|c|c|}
\hline
\textbf{Step sizes (dual variables)} & $\alpha_{\bar{v}} $ & $\alpha_{\underline{v}} $ & $\alpha^{pv}_{i, \bar{P}} $ & $\alpha^{pv}_{i,\underline{P}} $ & $\alpha^{pv}_{i,\bar{S}}$ & $\alpha^{bt}_{i,\bar{P}} $ & $\alpha^{bt}_{i,\underline{P}} $ & $\alpha^{bt}_{i, \bar{S}}$ \TTop \BBot   \\ \hline
\textbf{Value }    & $10$ &  $10$ &  $1\times10^{-3}$ & $1\times10^{-3}$  & $1\times10^{-3}$  & $2\times10^{-4}$  & $2\times10^{-4}$  & $1\times10^{-3}$  \TTop \BBot  \\ \hline
\end{tabular}

\label{t:parameter}
\end{table*}

In simulation, the load and PV data were obtained from the real data in California \cite{banham2013}. They were smoothed and scaled to an appropriate magnitude for the considered distribution system. The total demand and available PV generation in a typical day are given in Fig. \ref{fig:loadpv}. 



The described distribution system was simulated using Matpower \cite{Matpower}, from where the feeder head active power and node voltages are obtained as system outputs. The output measurements were obtained by adding Gaussian noises to the  simulated outputs; in particular,
\begin{equation}
\hat{\by}^{(k)}_i = \by^{(k)}_i +  W \by^{(k)}_i
\label{e:measureNoise}
\end{equation}
where the random variable $W$ follows the normal distribution $\mathcal{N}(0,\sigma^2)$. In simulations, unless otherwise stated, $\sigma = 0.001$, which is consistent with the standard phasor measurement units noise values \cite{Xie2014,Brown2016}.

Algorithm parameters were configured as shown in Table \ref{t:parameter}.  For each local controller, the sine signal was chosen as the perturbation signal:
\[
\probe_n(t) = \sqrt{2} sin(2 \pi f_n t), \quad t \in [0 \;\; 24 hr]
\]
where the time step $\Delta t =1$ second;  $\epsy = 0.001  \sqrt{2} \, S_{base}$, with base power $S_{base} = 23.04$ MVA; and the frequency $f_n$ was uniquely chosen from the range $[1/7.1, \, 1/26]$ Hz.

\subsection{Simulation Results}


The results described here are based on the 24-hour simulation with the load and PV data given in Fig. \ref{fig:loadpv}.
It is observed from Fig. \ref{fig:Ptrack} that the feeder head power nicely tracks the desired reference power and the DERs can respond quickly to unforeseen changes of the reference signal. The  subplots on the bottom of Fig. \ref{fig:Ptrack} show the detailed response when the reference signal was changed. The power tracking is less accurate during $[15, \, 21 ]$ hours, because the power tracking  was conflicting with another objective of voltage regulation. 
\begin{figure}[h]
    \centering
    \includegraphics[width=.5\textwidth]{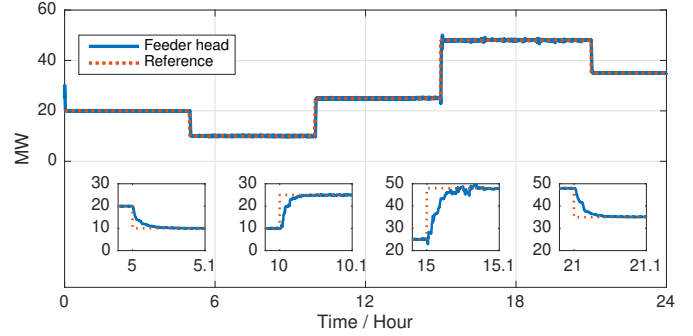}
    \caption{Performance of feeder head power tracking achieved by the gradient-free approach.}
    \label{fig:Ptrack}
\end{figure}

The performance of voltage regulation is illustrated in Fig. \ref{fig:voltage}, which provides the comparison of node voltages with and without voltage regulation. The voltage lower limit was set to 0.96 p.u. At the time around $15$ hours, when some node voltages dropped below the limit, they were recovered quickly to normal values. In contrast, without regulation (no voltage limit), some node voltages stayed below 0.96 p.u. for the next 6 hours. 

\begin{figure}[!b]
    \centering
    \includegraphics[width=0.45\textwidth]{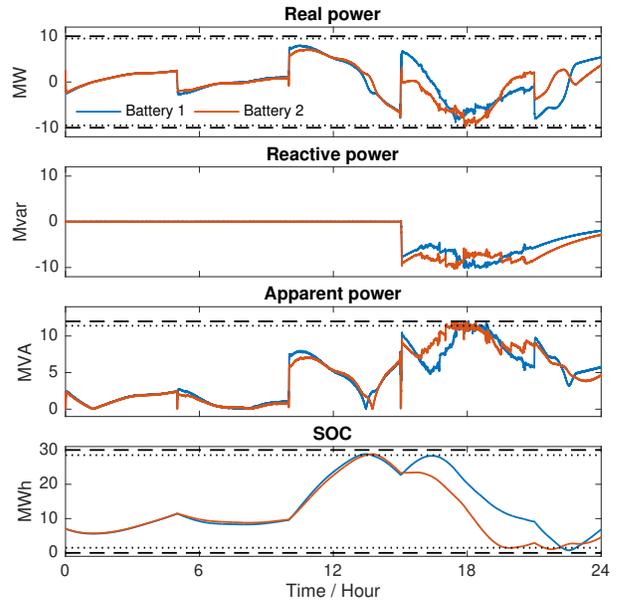}
    \caption{Control results of two battery systems. The dashed lines  represent hard physical constraints imposed on batteries. Corresponding to each  hard constraint, the soft constraint is slightly more conservative than  the hard constraint, shown as the dotted lines. Soft constraints enable smooth transition of control signals to avoid  hard-constraint projections.}
    \label{fig:battery}
\end{figure}

\begin{figure*}[h]
    \centering
    \includegraphics[width=1\textwidth]{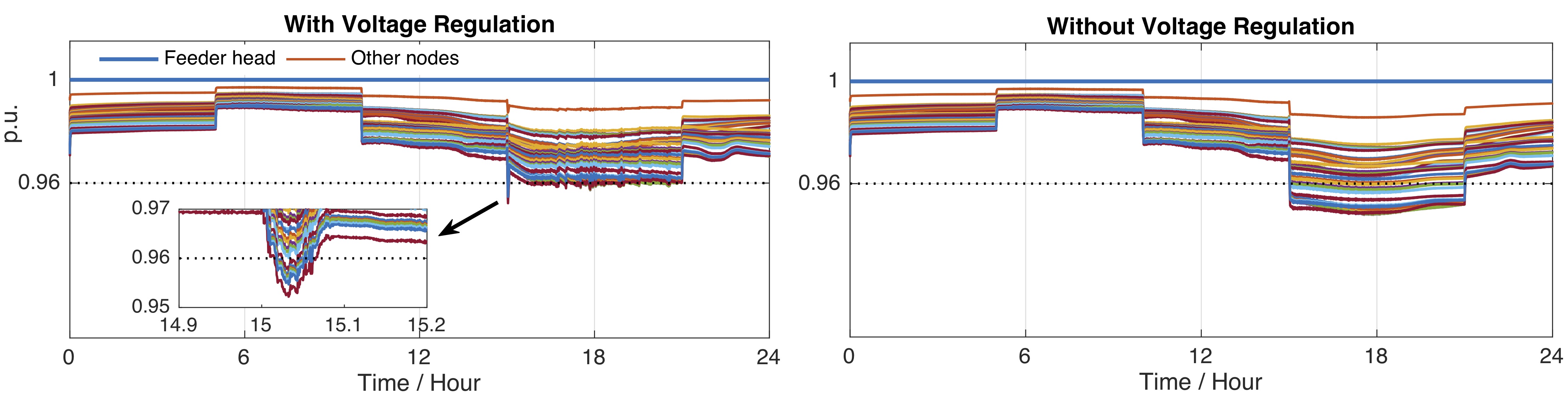}
    \caption{Improvement of voltage profile by the proposed model-free OPF algorithm.}
    \label{fig:voltage}
\end{figure*}

These grid services were provided through the control of battery and PV systems, whose behaviors are presented in Fig. \ref{fig:battery} and Fig. \ref{fig:PVctrl}, respectively. 
The feeder head power was controlled using the active power of DERs, mainly from batteries. The PV systems have its main objective of curtailment minimization, so the active PV power was merely increased because it is already near the maximal generation. But the active PV power can be largely curtailed when the system requested less power generation (e.g., at time 10, 15 hours). 
Because the active power of DERs is determined to track the reference feeder head power, the voltage regulation is mainly accomplished by the reactive power of DERs. At the time around 15 hours when some node voltages dropped below the lower limit, the reactive power from both batteries and PVs acted quickly to recover these voltages to the normal range.

\begin{figure}[!t]
    \centering
    \includegraphics[width=.45\textwidth]{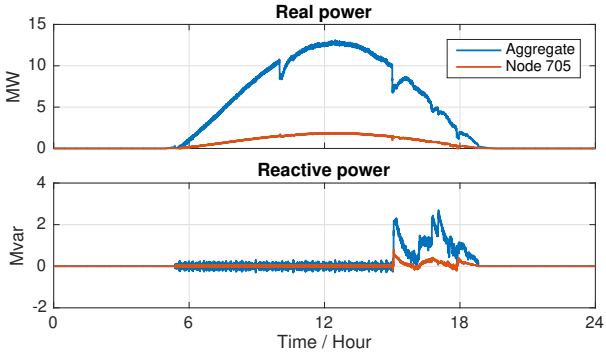}
   \caption{Results of PV power control.}    \label{fig:PVctrl}
\end{figure}


To study the impact of the measurement noise, we tested the proposed algorithm with different levels of noise, distinguished by the standard deviation of the noise signal in \eqref{e:measureNoise}. 
Two metrics are used to quantify system performance. The normalized root mean squared error (NRMSE) is used to define power tracking performance:
\[
    \text{NRMSE} = \sqrt{\frac{1}{K}\sum_{k=1}^K \left(\frac{P_{0}^{(k)} - P_{0}^{\bullet (k)}}{P_{0}^{\bullet(k)}} \right)^2 }
\]
The average voltage violation (AVV) is used to define voltage regulation performance:
\[
\text{AVV} = \frac{1}{NK}\sum_{i=1}^N \sum_{k=1}^K \left([\bv_i^{(k)}- \shortoverline{V}_i]_+ + [\shortunderline{V}_i - \bv_i^{(k)}]_+ \right)
\]
where $N$ and $K$ are the number of nodes and simulation time steps, respectively; the operator $[\;]_+$ projects a negative value to zero.

Fig. \ref{fig:perf_noise} shows simulation results of NRMSE and AVV. As the noise signal increases and crosses the point of $\sigma=1.6\times 10^{-3}$, the control performance degrades significantly. It is because  large measurement noise can suppress the exploration output and leads to inaccurate gradient estimation. One idea to address this issue is to increase the exploration signal to a reasonable magnitude. But the improvement space is limited because of practical system considerations. In that case, we can apply state estimation techniques to filter out measurement noises. It is beyond the scope of this paper and left to our future work.

\begin{figure}[!t]
    \centering
    \includegraphics[width=.5\textwidth]{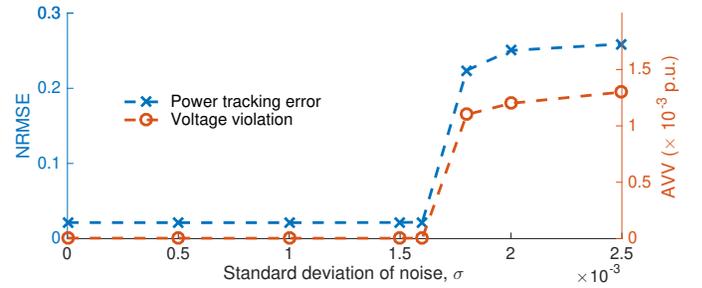}
   \caption{The performance is impacted by the measurement noise.}    \label{fig:perf_noise}
\end{figure}

\section{Conclusions and Future Work}
\label{sec:conclusion}

In this paper, we developed a model-free primal-dual method to track desired trajectories of networked systems. The algorithm leverages the zero-order deterministic estimation of the gradient with two function evaluations. We provided preliminary stability and tracking results, and illustrated an application of the method to real-time OPF in power systems.

Analysis of the exact iteration \eqref{eqn:primal}-\eqref{eqn:dual}, with the projection performed at every time step, remains an interesting research topic. The promising direction here is to analyze the time-varying projected dynamics of the form \eqref{eqn:gf_cont_proj}. Also, analysis of the asynchronous model-free algorithm is important from the practical perspective. Finally, on the application side, it would be interesting to show the performance of the algorithm on a more dynamic scenario, including topology reconfigurations (e.g., due to natural disasters and attacks on the grid).

\bibliographystyle{IEEEtran}

\end{document}